\documentclass[11pt, reqno]{amsart}

\usepackage{amssymb, amscd, latexsym}
\usepackage{enumitem}
\usepackage{amsmath}
\usepackage{amsthm}
\usepackage{bm}
\usepackage{graphicx}
\usepackage[all]{xy}
\usepackage{float}
\usepackage{mathrsfs}
\usepackage{pgf,tikz}
\usetikzlibrary{arrows}
\usepackage{array}
\usepackage{arydshln}
\usepackage{hyperref}
\usepackage{pstricks}
\hypersetup{
pdftitle={Research Paper},
pdfauthor={Safyan Ahmad},
pdfsubject={CMizations},
pdfcreator={Safyan Ahmad},
colorlinks,
linkcolor=blue,
citecolor=magenta,
anchorcolor=red,
bookmarksopen,
urlcolor=red,
filecolor=red,
pdfpagetransition={Wipe}
}
\psset{unit=0.7cm,linewidth=0.8pt,arrowsize=2.5pt 4}

\newpsstyle{fatline}{linewidth=1.5pt}
\newpsstyle{fyp}{fillstyle=solid,fillcolor=verylight}
\definecolor{verylight}{gray}{0.97}
\definecolor{light}{gray}{0.9}
\definecolor{medium}{gray}{0.85}

\newtheorem{Theorem}{Theorem}[section]

\newtheorem{Corollary}[Theorem]{Corollary}
\newtheorem{Proposition}[Theorem]{Proposition}
\newtheorem{Remark}[Theorem]{Remark}

\newtheorem{Example}[Theorem]{Example}

\newtheorem{Definition}[Theorem]{Definition}

\def\cocoa{{\hbox{\rm C\kern-.13em o\kern-.07em C\kern-.13em o\kern-.15em A}}}

\def\implies{\ifmmode\Rightarrow \else
        \unskip${}\Rightarrow{}$\ignorespaces\fi}

\begin{document}

    \title[Cohen--Macaulay Hybrid Graphs]{Cohen--Macaulay Hybrid Graphs}
    \author[Safyan, Imran \& Fazal]{Safyan Ahmad, Imran Anwar, Fazal Abbas}

        \address{Safyan Ahmad, Abdus Salam School of Mathematical Sciences, GC University Lahore, Pakistan.
} \email{safyank@gmail.com}

\address{Imran Anwar, Abdus Salam School of Mathematical Sciences, GC University Lahore, Pakistan.
} \email{imrananwar@sms.edu.pk}

\address{Fazal Abbas, Abdus Salam School of Mathematical Sciences, GC University Lahore, Pakistan.
} \email{fazalabbas307@yahoo.com}

 \subjclass[2010]{13A02, 13C14, 13D02, 13P10}

    \keywords{Cohen--Macaulay Graphs, Unmixed Graphs, Chordal Graphs, Clique Complex,  Shellable Simplicial Complexes}

\begin{abstract}
We introduce a new family of graphs, namely, {\em hybrid graphs}.
There are infinitely many hybrid graphs associated to a single
graph. We show that every hybrid graph associated a given graph is
Cohen--Macaulay. Furthermore, we show that every Cohen--Macaulay
chordal graph is a hybrid graph.
\end{abstract}

\maketitle

\section*{introduction}
Let $G$ be the simple graph on the vertex set $[n]$ and the edge set $E(G)$. We identify the vertex $i$ with the variable $x_i$ and consider the polynomial ring $S=K[x_1, \ldots, x_n]$ over an arbitrary field $K$. Let $I(G)=(x_ix_j: \{i,j\} \in E(G)) \subset S$ denotes the edge ideal of $G$. We say that the graph $G$ is Cohen--Macaulay if $S \slash I(G)$ is Cohen--Macaulay. Using the Stanley--Reisner correspondence, one can associate to $G$ the simplicial complex $\Delta_G$, whose faces are the independent subsets of $[n]$ and whose Stanley--Reisner ideal coincides with the edge ideal of $G$, i.e. $I_{\Delta_G} = I(G)$.\\
According to \cite{HHZ}, it is unlikely to have a general classification of Cohen--Macaulay graphs. Therefore, it is natural to study this question for some special classes of graphs. For instance, Villarreal in \cite{V} gave classification of all Cohen--Macaulay trees. Herzog, Hibi and Zheng \cite{HHZ} classified all Cohen--Macaulay chordal graphs later  Herzog and Hibi \cite{HH} did for all Cohen--Macaulay bipartite graphs.\\
The potential aim of this paper is to generate new graphs $G'$ from a given graph $G$ with the property that $G'$ is always Cohen--Macaulay. The classification and construction of Cohen--Macaulay graphs is one of the central problems and  enjoys rich literature, for instance  \cite{CN}, \cite{DE}, \cite{EV}, \cite{FA}, \cite{VTW} and \cite{W}. Here, we generalizes the notion of "whiskering", introduced by Villareal\cite{V}. By a whisker to a vertex $x$, we mean to add a new vertex $y \not\in [n]$ and an edge $\{x,y\}$. Villareal proved that if we add whiskers to every vertex of any graph $G$, the resulting graph $G'$ is always Cohen--Macaulay. Cook and Nagel \cite{CN}, extends this work and defined clique-whiskering. They proved that $G'$ obtained by clique-whiskering is always Cohen--Macaulay. Recently, J. Biermann and A. V. Tuyl \cite{BV} defined $s-$coloring $\chi$ on a simplicial complex $\Delta$ to construct a Cohen--Macaulay simplicial complex $\Delta_{\chi}$, generalizing the previous constructions for simplicial complexes. More recent, A. M. Liu and T. Wu \cite{MT} generalized this notion of clique-whiskering for obtaining families of sequentially Cohen--Macaulay graphs.\\
In this paper, we give a construction of hybrid graphs
(see~\ref{mainsec}).  The well--known constructions mentioned in
\cite{CN} and \cite{V}  appear to be particular cases of our
construction. Moreover, by using the construction discussed in
\cite{V}, one may obtain one Cohen-Macualy Graph corresponding to a
given simple graph $G$. Similarly, by using the construction,
corresponding to each clique partition of vertex set, one may obtain
one Cohen--Macaulay graph. But in our case, corresponding to any
simple graph $G$ and a given clique partition of the vertex set of
$G$, one may obtain infinitely many Cohen--Macaulay graphs.  We call
these new graphs as the hybrid graphs associated to $G$, here is our
main result.
\begin{Theorem}[Theorem~\ref{thm-cohen-macaulay-graphs}]
Let $G$ be a graph on the vertex set $[n]$ and $A_1, \ldots, A_r$ be
a partition of $[n]$ into disjoint subsets such that $A_i$ is a
clique in $G$. Let $B_i=\{y_{i,1}, \ldots, y_{i,{s_i}} \}$ and
$G'=G_{{A_1}, \ldots, {A_r}}^{{B_1}, \ldots, {B_r}}$ be a hybrid
graph of $G$ with respect to $B_1, \ldots, B_r$. Let $\Delta'$ be
the independence complex of $G'$ and $S=\mathbb{K}[\{x_i \colon \;
i=1, \ldots, n\} \cup \{y_{i,j} \colon \; i=1, \ldots, r, \; j=1,
\ldots, s_i\}]$ be the polynomial ring. Then,
\begin{itemize}
\item[\rm (i)] $\Delta'$ is pure shellable of dimension $r-1$.
\item[\rm (ii)] The ring $\frac{S}{I \left( G' \right)}$ is Cohen-Macaulay of dimension $r$.
\end{itemize}
\end{Theorem}
As a quick consequence of the above theorem, we obtain following
well-known results.
\begin{Corollary}[Villareal]
\label{whisker} Suppose $G$ is a graph and let $G'$ be the graph
obtained by adding a whisker at every vertex $v \in G$. Then the
ideal $I(G')$ is Cohen--Macaulay.
\end{Corollary}
\begin{Corollary}[Cook and Nagel]
Let $\pi=\{W_1, \ldots, W_t\}$ be a clique vertex partition of $G$.
Then $I(G^{\pi})$ is Cohen--Macaulay.
\end{Corollary}
Finally, we give a complete characterization all Cohen--Macaulay
chordal graphs by adding one more equivalent condition to a well
known characterization of Cohen--Macaulay chordal graphs due to
Herzog, Hibi and Zehng \cite{HHZ},
\begin{Corollary}
Let $K$ be a field and $G$ be a chordal graph  on the vertex set
$[n]$. Let $F_1, \ldots, F_m$ be the facets of $\Delta(G)$ which
admit a free vertex. Then the following are equivalent:
\begin{enumerate}
  \item $G$ is Cohen--Macaulay;
  \item $G$ is Cohen--Macaulay over $K$;
  \item $G$ is unmixed;
  \item $[n]$ is the disjoint union of $F_1, \ldots, F_m$.
  \item $G$ is a hybrid graph.
\end{enumerate}
\end{Corollary}

\section{Preliminaries} \label{prelim}
In this section, we recall some necessary definitions and results.

\subsection{Simplicial complexes}
A {\em simplicial complex} $\Delta$ on the vertex set $V=\{v_1,
\ldots, v_n\}$ is a collection of subsets of $V$ such that $\{ v_i
\} \in \Delta$  for all $i$ and, $F \in \Delta$ implies that all
subsets of $F$ are also in $\Delta$. The elements of $\Delta$ are
called {\em faces} and the maximal faces under inclusion are called
{\em facets} of $\Delta$. We denote by $\mathcal{F}(\Delta)$ the set
of facets of $\Delta$. We say that a simplicial complex is {\em
pure} if all its facets have the same cardinality.
 The {\em dimension} of a face $F$ is $\dim F = |F|-1$, where $|F|$ denotes the cardinality of $F$. A simplicial complex is called {\em pure} if all its facets have the same dimension.
The \textit{dimension} of $\Delta$, $\dim(\Delta)$, is defined as:
$$\dim (\Delta) = \max\{\dim F \colon F \in \Delta \}.$$
Given a simplicial complex $\Delta$ on the vertex set $\{v_{1},
\ldots, v_{n}\}$. For $F \subseteq \{v_{1}, \ldots, v_{n} \}$ let
$\textbf{x}_F=\prod_{v_i \in F}{x_i}$, and let
$\textbf{x}_\varnothing = 1$. The \emph{non-face ideal} or the
\emph{Stanley-Reisner ideal} of $\Delta$, denoted by $I_\Delta$, is
an ideal of $S$ generated by square-free monomials $\textbf{x}_F$,
where $F \not\in \Delta$.

\begin{Definition}[Shellable simplicial complexes] \label{definition of shellability}
A simplicial complex $\Delta$ is called {\em shellable} if there is
a total order of the facets of  $\Delta$, say $F_1, \ldots, F_t$,
such that $\langle F_1, \ldots, F_{i-1} \rangle \cap \langle F_i
\rangle$ is generated by a non-empty set of maximal proper faces of
$F_i$ for $2 \leq i \leq t$. Any such order is called a {\em
shelling order} of $\Delta$.
\end{Definition}

To say $F_1, F_2, \ldots, F_t$ is a shelling order of $\Delta$; it
is equivalent to saying
that for all $F_i$ and all $F_j < F_i$, there exists $x \in F_i \setminus F_j$ and $F_k < F_i$ such that $F_i \setminus F_k = \{x\}$.\\
The class of shellable simplicial complexes is important due to the
following result.
\begin{Theorem}\label{shellable-is-cm}[Theorem 8.2.6, \cite{HHBook}]
A pure shellable simplicial complex is Cohen--Macaulay over any
arbitrary field.
\end{Theorem}
\subsection*{Graphs}
Throughout in this paper, $G$ will denote a simple graph on $[n]$ vertices which means $G$ has no loops or multiple edges. Given a subset $W$ of $[n]$, we define the {\em induced subgraph} of $G$ on $W$ to be the subgraph $G_W$ on $W$ consisting of those edges $\{i,j\} \in E(G)$ with $\{i,j\} \subset W$. A {\em walk of length $m$} in $G$ is a sequence of vertices $\{i_0, \ldots, i_m\}$ such that $\{i_{j-1}, i_j\}$ are edges in $G$ . A {\em cycle of length $m$} is a closed walk $\{i_0, \ldots, i_m\}$ in which $n \geq 3$ and the vertices $i_1, \ldots, i_m$ are distinct. A graph $G$ on $[n]$ is {\em connected} if, for any two vertices $i$ and $j$, there is a walk between $i$ and $j$. A connected graph without cycles is said to be a {\em tree}. The {\em complete graph} $K_m$ has every pair of its $m$ vertices adjacent.\\
A {\em chord} of a cycle $C$ is an edge $\{i,j\}$ of $G$ such that
$i$ and $j$ are vertices of $C$ with $\{i,j\} \not\in E(C)$. A graph
is said to be {\em chordal} graph if each of its cycles of length
$>3$ has a chord, obviously every tree is a chordal graph. A subset
$C$ of $V(G)$ is called a {\em clique} of $G$ if induced subgraph
$G_C$ is complete. The {\em clique complex} $\Delta(G)$ of a finite
graph $G$  on $V(G)$ is a simplicial complex whose faces are the
cliques of $G$. For a finite simple graph $G$ on $n$ vertices, one
may associate a square-free monomial ideal $I(G)$ in $S =K[x_1,
\ldots, x_n]$, namely, \em edge ideal \em of $G$ defined as,
\[
I(G)=(x_ix_j: \{i,j\} \in E(G))
\]

We say that $G$ is  Cohen--Macaulay over the field $K$, if the
associated quotient ring $S \slash I(G)$ is Cohen--Macaulay. A
subset $C \in V(G)$ is called a {\em vertex cover} of $G$ if $C\cap
E\neq\emptyset$ for all $E\in E(G)$. A vertex cover $C$ is {\em
minimal} if no proper subset of $C$ is a vertex cover of $G$ and if
all minimal vertex covers of $G$ have same cardinality, then we say
that $G$ is {\em unmixed}. All Cohen--Macaulay graphs are unmixed
but not vice versa. The following result characterizes all
Cohen--Macaulay chordal graphs.
\begin{Theorem}\cite[Theorem 2.1]{HHZ}
\label{HHZ} Let $K$ be a field and $G$ be a chordal graph  on the
vertex set $[n]$. Let $F_1, \ldots, F_m$ be the facets of
$\Delta(G)$ which admit a free vertex. Then the following are
equivalent:
\begin{enumerate}
  \item $G$ is Cohen--Macaulay;
  \item $G$ is Cohen--Macaulay over $K$;
  \item $G$ is unmixed;
  \item $[n]$ is the disjoint union of $F_1, \ldots, F_m$.
\end{enumerate}
\end{Theorem}

\section{Main result}
\label{mainsec} The following definition is essential to understand
the underlined construction.
\begin{Definition}
Let $G$ be a graph on the vertex set $[n]$ and $A_1, \ldots, A_r$ be
a partition of $[n]$ into disjoint subsets such that $A_i$ is a
clique in $G$($A_i$ can be empty set). For each $i=1, \ldots, r$,
let $B_i=\{y_{i,1}, \ldots, y_{i,{s_i}} \}$ be a non-empty set.
Define the graph $G_{{A_1}, \ldots, {A_r}}^{{B_1}, \ldots, {B_r}}$
as follows:
\begin{equation}
G_{{A_1}, \ldots, {A_r}}^{{B_1}, \ldots, {B_r}} = G \cup \left(
\mathop{\cup}\limits_{i=1}^r \{F \subset A_i \cup B_i \colon \quad
|F|=2 \} \right).
\end{equation}
We call the graph $G_{{A_1}, \ldots, {A_r}}^{{B_1}, \ldots, {B_r}}$,
the hybrid graph of $G$ with respect to $B_1, \ldots, B_r$.
\end{Definition}
\begin{Remark}
Corresponding to each partition of the vertex set, we have
infinitely many choices to choose $B_i$'s thus there are infinitely
many hybrid graphs corresponding to any given graph.
\end{Remark}

Let $G$ be a graph and $G'= G_{{A_1}, \ldots, {A_r}}^{{B_1}, \ldots,
{B_r}}$ be the hybrid graph of $G$ with respect to $B_1, \ldots,
B_r$. Let $\Delta$ and $\Delta'$ be the independence complexes of
$G$ and $G'$ respectively. Let $$S=K[\{x_i \colon \; i=1, \ldots,
n\} \cup \{y_{i,j} \colon \; i=1, \ldots, r, \; j=1, \ldots,
s_i\}]$$ be the polynomial ring. Let us define the ordering on the
variables:
\begin{equation}
\label{ordering on x and y} x_1 > \cdots > x_n > y_{1,1}>\cdots
>y_{1,s_1} > \cdots > y_{r,1}> \cdots > y_{r,s_r}
\end{equation}
As the facets of $\Delta'$ are maximal independent sets in $G'$, it
is easy to see that the induced subgraph of $G'$ on  $A_i\cup B_i$
is a complete graph. Thus in an independent set, we can select at
most one element from $A_i \cup B_i$ for all $i$. In other words, if
$T$ be a facet of $\Delta'$, then $|T \cap (A_i \cup B_i)|=1$ for
all $i$. Let $F = T \cap [n]$, then $F$ will be an independent set
in $G$ and hence a face of $\Delta$. Let us consider $B =
\cup_{i=1}^ r B_i$ and  $F' = T \cap B$, then $T=F \cup F'$ and $F
\cap F'= \emptyset$. It is easy to note that $F' = \bigcup_{j: A_j
\cap F= \emptyset} \{y_{j,k_j}\}$ for some $1 \leq k_j \leq s_j$.
Let us record this simple observation in the following proposition.
\begin{Proposition}
\label{shape of facets} The independence complex $\Delta'$ of the
hybrid graph $G'$ is pure and every facet of $\Delta'$ is of the
form $F \cup F'$, where $F$ is a face of $\Delta$ and $F' =
\bigcup_{j: A_j \cap F= \emptyset} \{y_{j,k_j}\}$ for some $1 \leq
k_j \leq s_j$.

\end{Proposition}

Note that there are $\prod _{j: A_j \cap F= \emptyset}|s_j|$ choices
for $F'$, thus corresponding to each face $F$ of $\Delta$, there is
a block of facets of $\Delta'$. Here, we explain this fact through
the following example.
\begin{Example}
Consider the graph $G$ with vertex set $V(G)=\{1,2,3,4\}$ and edge
set $E(G)=\{\{1,2\},\{1,3\},\{2,3\},\{2,4\},\{3,4\}\}$. Consider the
vertex partition as $V(G)=\{1\}\cup \{2\}\cup \{3\}\cup \{4\}$ and
take $B_1=\{5,6\}, B_2=\{7\}, B_3=\{8\}, B_4=\{9,10,11\}$ then
$G'=G_{{A_1}, \ldots, {A_4}}^{{B_1}, \ldots, {B_4}}$ will be,

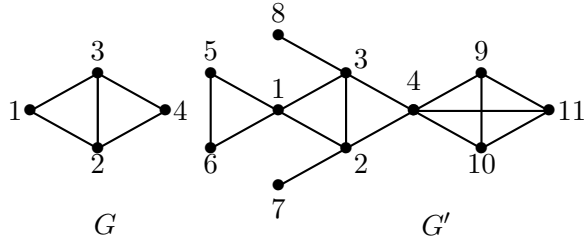
\begin{center}
\psset{unit=1cm}
\begin{pspicture}(0,0)(7,3)
\pspolygon[style=fyp, fillcolor=white](0.5,1.5)(1.4,2)(1.4,1)
 \pspolygon[style=fyp, fillcolor=white](1.4,1)(1.4,2)(2.3,1.5)
  \rput(0.3,1.5){$1$}
 \rput(1.4,2.3){$3$}
 \rput(1.4,0.8){$2$}
 \rput(2.5,1.5){$4$}
 \rput(0.5,1.5){$\bullet$}
 \rput(1.4,2){$\bullet$}
 \rput(1.4,1){$\bullet$}
 \rput(2.3,1.5){$\bullet$}
\rput(1.5,0){$G$} \pspolygon[style=fyp,
fillcolor=white](3.8,1.5)(4.7,2)(4.7,1)
 \pspolygon[style=fyp, fillcolor=white](4.7,1)(4.7,2)(5.6,1.5)
 \psline(2.9,2)(3.8,1.5)
 \psline(4.7,1)(3.8,0.5)
 \psline(4.7,2)(3.8,2.5)
 \psline(2.9,1)(3.8,1.5)
 \psline(2.9,2)(2.9,1)
 \pspolygon(5.6,1.5)(6.5,2)(7.4,1.5)(6.5,1)
 \psline(6.5,2)(6.5,1)
 \psline(5.6,1.5)(7.4,1.5)

 \rput(3.8,1.8){$1$}
 \rput(4.9,2.2){$3$}
 \rput(4.9,0.8){$2$}
 \rput(5.6,1.9){$4$}
 \rput(2.9,2.3){$5$}
 \rput(2.9,0.8){$6$}
 \rput(3.8,0.2){$7$}
 \rput(3.8,2.8){$8$}
 \rput(6.5,2.3){$9$}
 \rput(6.5,0.8){$10$}
 \rput(7.7,1.5){$11$}

 \rput(3.8,1.5){$\bullet$}
 \rput(4.7,2){$\bullet$}
 \rput(4.7,1){$\bullet$}
 \rput(5.6,1.5){$\bullet$}
 \rput(2.9,2){$\bullet$}
 \rput(2.9,1){$\bullet$}
 \rput(3.8,2.5){$\bullet$}
 \rput(5.6,1.5){$\bullet$}
 \rput(6.5,2){$\bullet$}
 \rput(7.4,1.5){$\bullet$}
 \rput(6.5,1){$\bullet$}
 \rput(3.8,0.5){$\bullet$}
 \rput(5.9,0){$G'$}
\end{pspicture}
\end{center}
The facets of $\Delta'$ are shown in Table A,
\begin{table}[h!]
  \begin{center}
  \label{table1}
  \caption{A}
  \begin{tabular}{l|l}
  \textbf{Faces of $\Delta$} &  \textbf{Corresponding facets of $\Delta'$} \\
  \hline
  $\emptyset$ & $\{5,7,8,9\},\{5,7,8,10\},\{5,7,8,11\},\{6,7,8,9\},\{6,7,8,10\},\{6,7,8,11\}$ \\
  $\{1\}$ & $\{1\}\cup\{7,8,9\}, \{1\}\cup\{7,8,10\}, \{1\}\cup\{7,8,11\} $ \\
  $\{2\}$ & $\{2\}\cup\{5,8,9\}, \{2\}\cup\{5,8,10\}, \{2\}\cup\{5,8,11\}, \{2\}\cup\{6,8,9\}, \{2\}\cup\{6,8,10\}, \{2\}\cup\{6,8,11\}$ \\
  $\{3\}$ & $\{3\}\cup\{5,7,9\}, \{3\}\cup\{5,7,10\}, \{3\}\cup\{5,7,11\}, \{3\}\cup\{6,7,9\}, \{3\}\cup\{6,7,10\}, \{3\}\cup\{6,7,11\}$ \\
  $\{4\}$ & $\{4\}\cup\{5,7,8\}, \{4\}\cup\{6,7,8\}$ \\
  $\{1,4\}$ & $\{1,2\}\cup\{7,8\}$
  \end{tabular}
  \end{center}
\end{table}

If we consider another partition of vertex set as $V(G)=\{1,3\}\cup
\{2\} \cup \{4\}$ and $B_1=\{5\}, B_2=\{6\}, B_3=\{7,8,9\}$, then
$G'=G_{{A_1}, \ldots, {A_3}}^{{B_1}, \ldots, {B_3}}$ will be,
\begin{center}
\psset{unit=1cm}
\begin{pspicture}(0,0)(7,3)
\pspolygon[style=fyp, fillcolor=white](3.8,1.5)(4.7,2)(4.7,1)
 \pspolygon[style=fyp, fillcolor=white](4.7,1)(4.7,2)(5.6,1.5)
 \psline(3.8,1.5)(3.8,2.5)
 \psline(3.8,2.5)(4.7,2)
 \psline(3.8,0.5)(4.7,1)
 \pspolygon(5.6,1.5)(6.5,2)(7.4,1.5)(6.5,1)
 \psline(6.5,2)(6.5,1)
 \psline(5.6,1.5)(7.4,1.5)

 \rput(3.6,1.5){$1$}
 \rput(4.9,2.2){$3$}
 \rput(4.9,0.8){$2$}
 \rput(5.6,1.9){$4$}
 \rput(3.8,2.8){$5$}
 \rput(3.8,0.2){$6$}
 \rput(6.5,2.3){$7$}
 \rput(6.5,0.8){$8$}
 \rput(7.6,1.5){$9$}

 \rput(3.8,1.5){$\bullet$}
\rput(4.7,2){$\bullet$}
 \rput(4.7,1){$\bullet$}
 \rput(5.6,1.5){$\bullet$}
 \rput(3.8,2.5){$\bullet$}
 \rput(5.6,1.5){$\bullet$}
 \rput(6.5,2){$\bullet$}
 \rput(7.4,1.5){$\bullet$}
 \rput(6.5,1){$\bullet$}
 \rput(3.8,0.5){$\bullet$}
 \rput(5.9,0){$G'$}
\end{pspicture}
\end{center}
In this case, Table B describes the facets of $\Delta'$,
\begin{table}[h!]
  \begin{center}
  \label{table2}
  \caption{B}
  \begin{tabular}{l|l}
  \textbf{Faces of $\Delta$} &  \textbf{Corresponding facets of $\Delta'$} \\
  \hline
  $\emptyset$ & $\{5,6,7\},\{5,6,8\},\{5,6,9\}$ \\
  $\{1\}$ & $\{1\}\cup\{6,7\}, \{1\}\cup\{6,8\}, \{1\}\cup\{6,9\} $ \\
  $\{2\}$ & $\{2\}\cup\{5,7\}, \{2\}\cup\{5,8\}, \{2\}\cup\{5,9\}$ \\
  $\{3\}$ & $\{3\}\cup\{6,7\}, \{3\}\cup\{6,8\}, \{3\}\cup\{6,9\}$ \\
  $\{4\}$ & $\{4\}\cup\{5,6\}$ \\
  $\{1,4\}$ & $\{1,4\}\cup\{6\}$
  \end{tabular}
  \end{center}
\end{table}

 Lastly, assume the partition of vertex set as, $V(G)=\{1,2,3\} \cup \{4\}$ and $B_1=\{5\}, B_2=\{6,7\}$, then $G'=G_{{A_1}, \ldots, {A_2}}^{{B_1}, \ldots, {B_2}}$ will be,

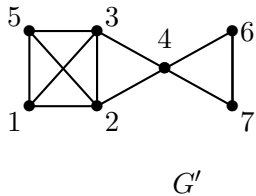
\begin{center}
\psset{unit=1cm}
\begin{pspicture}(0,0)(7,3)
\pspolygon[style=fyp, fillcolor=white](3.8,1)(4.7,2)(4.7,1)
 \pspolygon[style=fyp, fillcolor=white](4.7,1)(4.7,2)(5.6,1.5)
 \psline(3.8,2)(3.8,1)
 \psline(3.8,2)(4.7,2)
 \psline(3.8,2)(4.7,1)
 \pspolygon(5.6,1.5)(6.5,2)(6.5,1)
 \psline(6.5,2)(6.5,1)

 \rput(3.6,0.8){$1$}
 \rput(4.9,2.2){$3$}
 \rput(4.9,0.8){$2$}
 \rput(5.6,1.9){$4$}
 \rput(3.6,2.2){$5$}
 \rput(6.7,2){$6$}
 \rput(6.7,0.8){$7$}

 \rput(3.8,1){$\bullet$}
\rput(4.7,2){$\bullet$}
 \rput(4.7,1){$\bullet$}
 \rput(5.6,1.5){$\bullet$}
 \rput(3.8,2){$\bullet$}
 \rput(5.6,1.5){$\bullet$}
 \rput(6.5,2){$\bullet$}
 \rput(6.5,1){$\bullet$}
 \rput(5.9,0){$G'$}
\end{pspicture}
\end{center}
See the table C for complete list of facets of $\Delta'$ in this
case.
\begin{table}[h!]
  \begin{center}
  \label{table3}
  \caption{C}
  \begin{tabular}{l|l}
  \textbf{Faces of $\Delta$} &  \textbf{Corresponding facets of $\Delta'$} \\
  \hline
  $\emptyset$ & $\{5,6\},\{5,7\}$ \\
  $\{1\}$ & $\{1\}\cup\{6\}, \{1\}\cup\{7\}$ \\
  $\{2\}$ & $\{2\}\cup\{6\}, \{2\}\cup\{7\}$ \\
  $\{3\}$ & $\{3\}\cup\{6\}, \{3\}\cup\{7\}$ \\
  $\{4\}$ & $\{4\}\cup\{5\}$ \\
  $\{1,4\}$ & $\{1,4\}\cup\emptyset$
  \end{tabular}
  \end{center}
\end{table}
\end{Example}
Now, we present the main result of this paper.
\begin{Theorem}
\label{thm-cohen-macaulay-graphs} Let $G$ be a graph on the vertex
set $[n]$ and $A_1, \ldots, A_r$ be a partition of $[n]$ into
disjoint subsets such that $A_i$ is a clique in $G$. Let
$B_i=\{y_{i,1}, \ldots, y_{i,{s_i}} \}$ and $G'=G_{{A_1}, \ldots,
{A_r}}^{{B_1}, \ldots, {B_r}}$ be a hybrid graph of $G$ with respect
to $B_1, \ldots, B_r$. Let $\Delta'$ be the independence complex of
$G'$ and $S=K[\{x_i \colon \; i=1, \ldots, n\} \cup \{y_{i,j} \colon
\; i=1, \ldots, r, \; j=1, \ldots, s_i\}]$ be the polynomial ring.
Then,
\begin{itemize}
\item[\rm (i)] $\Delta'$ is pure shellable of dimension $r-1$.
\item[\rm (ii)] The ring $S\slash{I \left( G' \right)}$ is Cohen-Macaulay of dimension $r$.
\end{itemize}
\end{Theorem}
\begin{proof}
Lemma~\ref{shape of facets} guarantees that the independence complex
$\Delta'$ of $G'$ is pure and has dimension $r-1$, thus it is
sufficient to show that $\Delta'$ is shellable. From above
discussion, we know that corresponding to every face of $\Delta$,
there is a block of facets of $\Delta'$. Now we define an order the
facets of $\Delta'$ to show that $\Delta'$ is shellable.
\begin{quote}
We order the faces of $\Delta'$ in terms of increasing dimensions.
If two faces have same dimension, we order them by \eqref{ordering
on x and y}. Thus, associated to each face $F$ of $\Delta'$, we
consider the block associated to $F$, ordered as in \eqref{ordering
on x and y}.
\end{quote}
Let us assume $S$ and $T$ be two distinct facets of $\Delta'$ with $S < T$, here arises two cases:\\
\textbf{Case:1. When $S$ and $T$ belongs to different blocks} We can
write $S= F \cup F'$ and $T = G \cup G'$ where $F, G$ are different
faces of $\Delta$ and $F' = \bigcup_{j: A_j \cap F= \emptyset}
\{y_{j,k_j}\}$ for some $1 \leq k_j \leq s_j$, $G' = \bigcup_{j: A_j
\cap G= \emptyset} \{y_{j,p_j}\}$ for some $1 \leq p_j \leq s_j$.
\\ As $F \neq G$, we must have some $x_t \in G \setminus F$. Let $G_1 = G \setminus \{x_t\}$, then $G_1$ will also be a face of $\Delta$. As $G_1 \subset G$ so \[
\{j: A_j \cap G = \emptyset\} \subset \{j: A_j \cap G_1 =
\emptyset\}
\]
thus $G'\subset G_1':=\bigcup_{j: A_j \cap G= \emptyset} \{y_{j,p_j}\}$ for some $1 \leq p_j \leq s_j$. In fact, if $x_t \in A_p$ then $G_1' = G' \cup \{y_{p, k_p}\}$ for some $1 \leq k_p \leq s_p$. Let us take $T_1 = G_1 \cup G_1'$, then $T_1 < T$ with $T \setminus T_1 = \{x_t\}$.\\
\textbf{Case:2. When $S$ and $T$ belongs to same block} In this
case, $S = F \cup F'$ and $T = F \cup F''$. As $S \neq T$, we have
$F' \neq F''$ and $T \setminus S \neq \emptyset$. Let $l$ be the
least number, such that $y_{l, k_l} \in F'' \setminus F'$, thus
$y_{l, k_l} \in F''$ and $y_{l, k_l} \notin F'$ which further
implies the existence of a $y_{l, k_l '} \in F'$ for some  $1 \leq
k_l' \neq k_l \leq s_l$ with $y_{l, k_l '} < y_{l, k_l}$ as $S < T$.

If $y_{l, k_l}$ is the only vertex in $F'' \setminus F'$, we are
done, otherwise suppose $y_{m, k_m} \in F'' \setminus F'$. We have
ordered $F'$ and $F''$ as defined in \eqref{ordering on x and y} and
as we have assumed $l$ to be least such number, thus the first $l-1$
components in $F'$ and $F''$  will be the same. Thus $F'$ and $F''$
will be of the form,
\[
F' = \{ \ldots , y_{l, k_l '}, \ldots, y_{m,k_m '}, \ldots \}
\]
\[
F''= \{ \ldots , y_{l, k_l}, \ldots, y_{m,k_m}, \ldots \}
\]
where $y_{m,k_m '} \neq y_{m,k_m }$. Let us consider,
\[
F'''=\{ \ldots , y_{l, k_l '}, \ldots, y_{m,k_m}, \ldots \}
\]
and suppose $T_1 = F \cup F'''$, then $T_1$ will be a facet of
$\Delta'$ by Lemma~\ref{shape of facets} with $T_1 < T$. Note that
$y_{m,k_m} \notin F'' \setminus F'''$ and $y_{l,k_l} \in F''
\setminus F'''$, thus $y_{m,k_m} \notin T\setminus T_1$ and
$y_{l,k_l} \in T\setminus T_1$. If $y_{l, k_l}$ is the only element
in $T \setminus T_1$, we are done, otherwise we shall repeat the
same process. As we have finite element, this process will terminate
in finite steps and we shall have a $T_i < T$ such that $T \setminus
T_i =\{y_{l,k_l}\}$, as required.

\end{proof}
Here, we give the descriptive definition of a hybrid graph.
\begin{Definition}
A graph $G$ is said to be {\bf hybrid} if there exists some graph
$H$ such that $G$ is a hybrid graph associated to $H$.
\end{Definition}

  The following result shows that the whiskering of a graph, given by Villarreal in \cite{V} is a particular case of our construction.
\begin{Corollary}[Villareal]
\label{whisker} Suppose $G$ is a graph and let $G'$ be the graph
obtained by adding a whisker at every vertex $v \in G$. Then the
ideal $I(G')$ is Cohen--Macaulay.
\end{Corollary}
\begin{proof}
Suppose $V(G)=\{v_1, \ldots, v_n\}$ and consider the trivial clique
partition of $V(G)$ into singleton sets as, $V(G)=\{v_1\} \cup
\ldots \cup \{v_n\}$. If we take $B_i =\{y_i\}$ for all $1 \leq i
\leq n$, then $G' = G_{{A_1}, \ldots, {A_n}}^{{B_1}, \ldots,
{B_n}}$, thus $\Delta'$ is pure and $I(G')$ is Cohen--Macaulay by
Theorem~\ref{thm-cohen-macaulay-graphs}.
\end{proof}
R. Woodroof \cite{W} and  D. Cook and U. Nagel \cite{CN} generalized
the the concept of whiskering, and defined the terms
\textit{clique-whiskering} and \textit{fully clique-whiskering}
respectively. Recall from \cite{CN} that a {\em vertex
clique-partition} $\pi$ of a graph $G$ is a partition $\pi = \{W_1,
\ldots, W_t\}$ of $V(G)$ such that each subgraph induced on $W_i$ is
a nonempty clique, see \cite{W}. The clique-whiskering is particular
case of our construction.

\begin{Corollary}[Theorem 3.3, \cite{CN}]
\label{clique whiskering} Let $\pi=\{W_1, \ldots, W_t\}$ be a vertex
clique-partition of $G$. Then $I(G^{\pi})$ is Cohen--Macaulay.
\end{Corollary}
\begin{proof}
If $\pi=\{W_1, \ldots, W_t\}$ be a clique vertex partition of $G$.
Let us take $B_i=\{y_i\}$, singleton sets for all $i$. Then
$G^{\pi}=H_{{W_1}, \ldots, {W_t}}^{{B_1}, \ldots, {B_t}}$ is a
particular hybrid graph associated to $G$ and hence Cohen--Macaulay
by Theorem~\ref{thm-cohen-macaulay-graphs}.
\end{proof}
Herzog {\em et al.} \cite{HHZ} characterize Cohen--Macaulay chordal
graphs. One can see easily that every Cohen--Macaulay chordal graph
is in fact a hybrid graph associate to some graph. Thus, it
characterizes all Cohen--Macaulay chordal graphs. By a free vertex
in a simplicial complex, we mean a vertex which belongs to exactly
on facet of the simplcial complex.
\begin{Corollary}
\label{HHZcor} Let $K$ be a field and $G$ be a chordal graph  on the
vertex set $[n]$. Let $F_1, \ldots, F_m$ be the facets of
$\Delta(G)$ which admit a free vertex. Then the following are
equivalent:
\begin{enumerate}
  \item $G$ is Cohen--Macaulay;
  \item $G$ is Cohen--Macaulay over $K$;
  \item $G$ is unmixed;
  \item $[n]$ is the disjoint union of $F_1, \ldots, F_m$.
  \item $G$ is a hybrid graph.
\end{enumerate}
\end{Corollary}
\begin{proof}
$(4) \Rightarrow (5)$ As,
\[
[n]= F_1 \cup \ldots \cup F_m
\]
where $F_i$ are cliques of $\Delta(G)$ containing free vertices. Let $A_i$ and $B_i$ denote the non-free and free vertices of $F_i$ respectively. Let $A = \cup_{i=1} ^ r A_i$ and $H := G_A $ be the induced graph. Then $G = H_{{A_1}, \ldots, {A_m}}^{{B_1}, \ldots, {B_m}}$.\\
$(5) \Rightarrow (1)$ Follows from
Theorem~\ref{thm-cohen-macaulay-graphs}.
\end{proof}
In particular, one obtains
\begin{Corollary}[Corollary 6.3.5, \cite{V1}]
If $G$ is a tree then the following are equivalent:
\begin{enumerate}
  \item $G$ is Cohen--Macaulay.
  \item $G$ is unmixed.
  \item $G$ is a hybrid graph.
\end{enumerate}
\end{Corollary}
\begin{Remark}
In small graphs, diagrammatically it is quite easy to check whether
a given graph is hybrid or not. It is pertinent to mention that
there exist graphs that are not hybrid but still Cohen--Macaulay.
For example, $5-$cycle is not hybrid but Cohen--Macaulay.
\end{Remark}
\section*{Acknowledgement}
We would like to thank the Higher Education Commission, Pakistan and
Abdus Salam School of Mathematical Sciences, Lahore Pakistan for
supporting and facilitating this research. We would also like to
thank Ali Akbar Yazdan Pour for valuable remarks during his visit to
the Abdus Salam School of Mathematical Sciences and for proposing
the name "hybrid graph".

\end{document}